\documentclass[12pt]{amsart}
\usepackage{amsfonts}
\usepackage{amssymb}
\usepackage[letterpaper, left=2.5cm, right=2.5cm, top=2.5cm,
bottom=2.5cm,dvips]{geometry}
\usepackage{verbatim}
\usepackage[T1]{fontenc}
\usepackage{graphicx}
\usepackage{diagbox}
\usepackage{xcolor}
\usepackage{amsmath}
\usepackage{array,multirow}

\usepackage[backref]{hyperref}
\hypersetup{
	colorlinks,
	linkcolor={blue!60!black},
	citecolor={green!60!black},
	urlcolor={green!60!black}
}

\newtheorem{theorem}{Theorem}
\theoremstyle{plain}

\newtheorem{conjecture}{Conjecture}
\newtheorem{corollary}{Corollary}

\numberwithin{equation}{section}

\DeclareMathOperator{\arb}{arb}

\begin{document}
\title[Cooperative coloring of matroids]{Cooperative coloring of matroids}
\author{Tomasz Bartnicki}
\address{Stanis\l aw Staszic State University of Applied Sciences in Pi\l a, 64-920 Pi\l a, Poland}
\email{t.bartnicki@ans.pila.pl}

\author{Sebastian Czerwi\'{n}ski}
\address{Faculty of Mathematics, Computer Science, and Econometrics,
University of Zielona G\'{o}ra, 65-516 Zielona G\'{o}ra, Poland}
\email{s.czerwinski@wmie.uz.zgora.pl}

\author{Jaros\l aw Grytczuk}
\address{Faculty of Mathematics and Information Science, Warsaw University
of Technology, 00-662 Warsaw, Poland}
\email{grytczuk@tcs.uj.edu.pl}

\author{Zofia Miechowicz}
\address{Stanis\l aw Staszic State University of Applied Sciences in Pi\l a, 64-920 Pi\l a, Poland}
\email{z.miechowicz@ans.pila.pl}

\thanks{The third author was supported in part by Narodowe Centrum Nauki, grant 2020/37/B/ST1/03298.}

\begin{abstract}
Let $M_1,M_2,\ldots,M_k$ be a collection of matroids on the same ground set $E$. A coloring $c:E \rightarrow \{1,2,\ldots,k\}$ is called \emph{cooperative} if for every color $j$, the set of elements in color $j$ is independent in $M_j$. We prove that such coloring always exists provided that every matroid $M_j$ is itself $k$-colorable (the set $E$ can be split into at most $k$ independent sets of $M_j$). We derive this fact from a generalization of Seymour's list coloring theorem for matroids, which asserts that every $k$-colorable matroid is $k$-list colorable, too. We also point on some consequences for the game-theoretic variants of cooperative coloring of matroids.
\end{abstract}

\maketitle

\section{Introduction}
Let $M$ be a matroid on the set $E$ not containing loops. A coloring of $E$ is \emph{proper} if each color class is an independent set. The least number of colors needed is denoted by $\chi(M)$ and called the \emph{chromatic number} of $M$. If $M$ is a \emph{graphic} matroid of a graph $G$, then $\chi(M)$ coincides with the \emph{arboricity} of $G$---the least number of colors in an edge-coloring of $G$ with no monochromatic cycles. By the well known theorem of Edmonds \cite{Edmonds} (see \cite{Oxley}) we know that
\begin{equation}
	\chi(M)=\max_{A\subseteq E}\left\lceil \frac{|A|}{r(A)}\right\rceil,
\end{equation}
where $r(A)$ is the matroid \emph{rank} of $A$ in $M$. This shows that the matroid chromatic number $\chi(M)$ is a rather tame parameter as against the traditional chromatic number of a graph. This paradigm holds as well for other, more sophisticated coloring problems. Let us compare, for instance, the issue of \emph{choosability} in both, graphical and matroidal settings.

Suppose that each vertex $v$ of a graph $G$ is assigned a list $L(v)$ of colors. We say that $G$ is \emph{$L$-colorable} if there is a proper coloring of $G$ in which every vertex $v$ gets a color from its list $L(v)$. A graph $G$ is \emph{$k$-choosable} it it is $L$-colorable from arbitrary lists $L(v)$, with $|L(v)|=k$. The least $k$ for which $G$ is $k$-choosable is called the \emph{list chromatic number} of $G$, and denoted by $\chi_{\ell}(G)$. Clearly, these notions may be defined \emph{mutatis mutandis} for matroids, and we may speak of choosability and the list chromatic number $\chi_{\ell}(M)$ of a matroid $M$.

As is well known, the difference $\chi_{\ell}(G)-\chi(G)$ may be arbitrarily large already for bipartite graphs. However, Seymour proved \cite{Seymour} that in case of matroids there is actually no difference between the two parameters, that is, $\chi_{\ell}(M)=\chi(M)$ holds for every matroid $M$. The proof is a simple application of the Matroid Union Theorem (see \cite{Oxley}) (see the next section). Some inetersting extensions of this result were obtained by Laso\'{n} \cite{Lason1} and Laso\'{n} and Lubawski \cite{Lason-Lubawski}.

In this paper we generalize Seymour's result to \emph{systems} of matroids, which has consequences for the matroidal analog of \emph{cooperative} graph colorings, introduced by Aharoni, Holzman, Howard, and Spr\"{u}ssel \cite{Aharoni-H-H-S} (se also \cite{Aharoni-B-C-H-J}, \cite{Bradshaw}). Let $G_1,G_2,\ldots,G_r$ be a system of graphs on the same vertex set $V$. A coloring of $V$ by colors $\{1,2,\ldots,r\}$ is a \emph{cooperative} coloring if for every $i$, the subset of vertices in color $i$ is independent in the graph $G_i$. Similarly one may define a cooperative coloring of matroids.

Our main result states that every system of $k$ matroids on the common set $E$, each with chromatic number at most $k$, has a cooperative coloring.

\section{The result}

Seymour \cite{Seymour} proved that every matroid $M$ satisfies $\chi (M)= \chi_{\ell} (M)$. Using a similar argument based on the Matroid Union Theorem (see \cite{Oxley}), we prove an extension of this result to systems of matroids. Let us first recall the statement of the Matroid Union Theorem.

Let $M_1,M_2,\ldots,M_t$ be matroids on the ground sets $E_1,E_2,\ldots,E_t$, with the rank functions $r_1,r_2,\ldots,r_t$. Let $M$ be the \emph{union} of matroids $M_i$, that is, the unique matroid on $E=E_1\cup\cdots\cup E_t$ whose independent sets are all possible unions of the form $A_1\cup\cdots \cup A_t$, where $A_i$ is any independent set in $M_i$.

\begin{theorem}[Matroid Union Theorem]\label{Theorem MUT} Let $M_1,M_2,\ldots,M_t$ be a collection of matroids on the sets $E_1,E_2,\ldots,E_t$, with the rank functions $r_i$, and let $M$ be their union on the set $E=E_1\cup\cdots\cup E_t$ with the rank function $r$. Then, every set $A\subseteq E$ satisfies $$r(A)=\min_{X\subseteq A}\left(\sum_{i=1}^t r_i(X\cap E_i)+|A\setminus X|\right).$$
\end{theorem}

This theorem implies a simple criterion for matroid $k$-colorability.
\begin{corollary}\label{Corollary MUT}
	Let $M$ be a matroid on the set $E$ with the rank function $r$. Then, $\chi (M)\leqslant k$ if and only if every subset $X\subseteq E$ satisfies 
	\begin{equation}
		r(X)\geqslant \frac{1}{k}\cdot |X|.
	\end{equation}
\end{corollary}

Using these results we may prove the following theorem.

\begin{theorem}\label{Theorem Matorid System}
	Let $t\geqslant k$ and let $N_1,N_2,\ldots,N_t$ be matroids on the same ground set $E$ satisfying $\chi (N_i)\leqslant k$, for all $i=1,2,\ldots,t$. Suppose that to each element $e\in E$, a list $L(e)\subseteq \{1,2,\ldots,t\}$ with $|L(e)|=k$, is assigned. Then there is a coloring $f:E\rightarrow \{1,2,\ldots,t\}$ such that
	\begin{itemize}
		\item[(i)] $f(e)\in L(e)$, for every $e\in E$,
		\item[(ii)] $f^{-1}(i)$ is independent in $N_i$, for all $i=1,2,\ldots,t$.
	\end{itemize}
\end{theorem}
\begin{proof}Let $E_i=\{e\in E:i\in L(e)\}$ and let $M_i=N_i\arrowvert _{E_i}$ be the restriction of the matroid $N_i$ to the set $E_i$. Denote by $r_i$ the rank functions of matroids $M_i$ and let $M$ be their union with the rank function $r$. Clearly, all matroids $M_i$ also satisfy $\chi (M_i)\leqslant k$. By the Matroid Union Theorem, there is some $X\subseteq E$ for which 
	\begin{equation}
		r(E)=\sum_{i=1}^t r_i(X\cap E_i)+|E\setminus X|.
	\end{equation}
	By  Corollary \ref{Corollary MUT} of the Matroid Union Theorem we have
	\begin{equation}
		r_i (X\cap E_i)\geqslant \frac{1}{k}\cdot |X\cap E_i|,
	\end{equation}
	for all $i=1,2,\ldots,t$. Therefore we may write
	\begin{equation}
		r(E)\geqslant\frac{1}{k}\cdot \sum_{i=1}^t |X\cap E_i|+|E\setminus X|.
	\end{equation}
	On the other hand, since every list $L(e)$ has size $k$, we have
	\begin{equation}
		\sum_{i=1}^t |X\cap E_i|=k\cdot|X|.
	\end{equation}
	Hence, it follows that
	\begin{equation}
		r(E)\geqslant\frac{1}{k}\cdot k\cdot|X|+|E\setminus X|=|X|+|E\setminus X|=|E|.
	\end{equation}
	This means that $E$ is an independent set in the matroid $M$. Since $M$ is the union of matroids $M_i$, it follows that $E=A_1\cup\cdots\cup A_t$ with $A_i$ independent in $M_i=N_i\arrowvert_{E_i}$. Clearly we may assume that $A_i$ are pairwise disjoint. Now one may define the coloring $f$ so that $f(e)=i$ whenever $e\in A_i$, which completes the proof.
\end{proof}
The following corollary confirms the existence of a cooperative coloring for arbitrary systems of matroids with a common upper bound on the chromatic number. 
\begin{corollary}\label{Corollary Matroid System}
Every system of $k$-colorable matroids $M_1,M_2,\ldots, M_k$ on the same ground set $E$ has a cooperative coloring.
\end{corollary}
\begin{proof}
	It suffices to take $t=k$ and $L(e)=\{1,2,\ldots,k\}$ for every $e\in E$, in Theorem \ref{Theorem Matorid System}.
\end{proof}
Let us illustrate the above corollary with an example. Let $G$ be a simple graph with the set of edges $E$. Suppose that $\arb(G)=2$. Let $\pi$ be any permutation of the set $E$. Then, by Coraollary \ref{Corollary Matroid System}, there is a forest $F$ in $G$ such that $\pi(E\setminus F)$ is also a forest in $G$. Clearly, this can be generalized to graphs with arboricity $k\geqslant2$ and any set of $k-1$ permutations of $E$.

Let us also mention that our result on cooperative coloring can be extended to the following game-theoretic settings considered by Laso\'{n} in \cite{Lason0} and \cite{Lason2}. The board of the game is a fixed set $E$, together with a collection of matroids $M_1,M_2,\ldots, M_t$ on $E$, and a fixed set of colors, $C=\{1,2,\ldots,t\}$. There are two players, Ann and Ben, who collectively color the set $E$ accordingly to the rule that the set of elements in color $i$ must be independent in matroid $M_i$. In every round, Ann points on a chosen uncolored element $e\in E$ and Ben colors it by a chosen color. Ann wishes the whole set $E$ to be colored, while Ben tends to prevent that from happening. If Ann has a strategy guaranteeing that she achieves her goal, then we say that the system $M_1,M_2,\ldots,M_t$ has a \emph{cooperative indicated coloring}.

Using the result of Laso\'{n} \cite{Lason0} and our Corollary \ref{Corollary Matroid System} we get immediatelly the following statement.

\begin{corollary}\label{Corollary Indicated}
	Every system of $k$-colorable matroids $M_1,M_2,\ldots,M_k$ on the set $E$ has a cooperative indicated coloring.
\end{corollary}

In another game-theoretic variant, Ann and Ben color the vertices alternately by choosing in each round an uncolored element $e\in E$ and a color $i\in C$ for it. All other rules are the same (see \cite{Bartnicki-G-K-Z} for a historical survey). If Ann has a strategy to finish the game with coloring the whole set $E$, then we say that the system of matroids $M_1,M_2,\ldots,M_t$ has a \emph{cooperative game coloring}.

Using another result of Laso\'{n} \cite{Lason2} and Corollary \ref{Corollary Matroid System} we get the following statement.

\begin{corollary}\label{Corollary Game}
	Every system of $k$-colorable matroids $M_1,M_2,\ldots,M_{2k}$ on the set $E$ has a cooperative game coloring.
\end{corollary}
\begin{proof}
Indeed, let us split the system into two parts $M_1,M_2,\ldots,M_k$ and $M_{k+1},M_{k+2},\ldots,M_{2k}$. Each of these two systems satisfies the assertion of Corollary \ref{Corollary Matroid System}. Hence there exists sets $A_i$ such that $A_i$ is independent in $M_i$ and the union $\bigcup_{i=1}^{2k}A_i$ is a $2$-covering of the set $E$ (each element $e\in E$ belongs to exactly two sets $A_i$). The assertion follows now from Theorem 1 in \cite{Lason2}.
\end{proof}

\section{Open problems}

Let us conclude this short note with just one open problem concerning yet another game-theoretic variant of cooperative coloring. In the single graph setting it was introduced independently by Schauz \cite{Schauz} and Zhu \cite{Zhu}.

As before, the game is palyed between Ann and Ben on a given system of matroids. So, let $M_1,M_2,\ldots,M_t$ be a system of $k$-colorable matroids on the same ground set $E$. In the $i$-th round of the game, Ben picks a subset $B_i\subseteq E$, and then Ann colors a subset $A_i\subseteq B_i$ by color $i$ so that $A_i$ is an independent set of $M_i$. Each element $e\in E$ can be picked at most $k$ times unless it is finally colored. Ann wins if the whole set is colored with this constraint after at most $t$ rounds. If she has a winning strategy, then the system $M_1,M_2,\ldots,M_t$ is said to have a \emph{cooperative painting}.

\begin{conjecture}\label{Conjecture Online}
	Every system of $k$-colorable matroids $M_1,M_2,\ldots,M_k$ on the set $E$ has a cooperative painting.
\end{conjecture}

If $M_1=\cdots =M_k$, then the statement of the conjecture is true, as proved by Laso\'{n} and Lubawski in \cite{Lason-Lubawski}. Notice that this result is yet another extension of the theorem of Seymour on the choosability of matroids.

\end{document}